\documentclass[12pt]{article}
\usepackage{amssymb}
\usepackage{amsmath,amsthm}
\usepackage[latin1]{inputenc}
\usepackage{hyperref}
\usepackage{enumerate}
\usepackage{color}
\usepackage{graphicx}
\hypersetup{colorlinks=true, linkcolor=blue, citecolor=blue,
urlcolor=blue}

 \newtheorem{remark}{Remark}

 \newtheorem{lemma}[remark]{Lemma}
 \newtheorem{theorem}[remark]{Theorem}
 \newtheorem{proposition}[remark]{Proposition}
 \newtheorem{corollary}[remark]{Corollary}

\title{Analogies between the geodetic number and the Steiner number of some classes of graphs}

\author{ Ismael G. Yero$^{1}$ and Juan A.
Rodr\'{\i}guez-Vel\'{a}zquez$^{2}$\\
    \\
$^1${\small Departamento de Matem\'aticas, Escuela Polit\'ecnica Superior de Algeciras}\\
{\small Universidad de C\'adiz,} {\small
Av. Ram\'on Puyol s/n, 11202 Algeciras, Spain.} \\ {\small
ismael.gonzalez\@@uca.es}\\
$^2${\small Departament d'Enginyeria Inform\`atica i Matem\`atiques}\\
{\small Universitat Rovira i Virgili,}  {\small Av. Pa\"{\i}sos
Catalans 26, 43007 Tarragona, Spain.} \\{\small
juanalberto.rodriguez\@@urv.cat}
\\
}
\date{}

\begin{document}

\maketitle

\begin{abstract}
A set of vertices $S$ of a graph $G$ is a geodetic set of $G$ if
every vertex $v\not\in S$ lies on a shortest path between two
vertices of $S$.
The minimum cardinality of a geodetic set of $G$ is the geodetic
number of $G$ and it is denoted by $g(G)$. A Steiner set of $G$ is a set of vertices $W$ of
$G$ such that every vertex of $G$ belongs to the set of vertices of
a connected subgraph of minimum size containing the vertices of $W$. The minimum cardinality of a Steiner set of $G$ is the Steiner
number of $G$ and it is denoted by $s(G)$. Let $G$ and $H$ be two graphs and let  $n$ be the order of $G$.  The corona product $G\odot H$ is defined as the graph obtained from $G$ and $H$ by taking one copy of $G$ and $n$ copies of $H$ and  joining by an edge each vertex from the $i^{th}$-copy of $H$ with the $i^{th}$-vertex of $G$. We study the geodetic number and the Steiner number of corona product graphs.
We show that if  $G$ is a connected graph of order $n\ge 2$ and  $H$ is a
non complete graph, then $g(G\odot H)\le s(G\odot H)$, which partially solve the open problem presented in [\emph{Discrete Mathematics} \textbf{280} (2004)  259--263] related to characterize families of graphs $G$ satisfying that $g(G)\le s(G)$.
\end{abstract}



\section{Introduction}

The Steiner distance of a set of vertices of a graph was introduced
as a generalization of the distance between two vertices
\cite{chartrand-first}. In this sense, Steiner sets in graphs could
be understood as a generalization of geodetic sets in graphs.
Nevertheless, its relationship is not exactly obvious. Some of the
primary results in this topic were presented in
\cite{chartrand-zhang}, where the authors tried to obtain a result
relating geodetic sets and Steiner sets. That is, they tried to show
that every Steiner set of a graph is also a geodetic set.
Fortunately, the author of \cite{pelayo-comment} showed by a
counterexample that not every Steiner set of a graph  is a
geodetic set,  and it was pointed out an open question related
to characterizing those graphs satisfying that every Steiner set is
geodetic or vice versa. Some relationships between Steiner sets and
geodetic sets  were obtained in
\cite{pelayo-1,chartrand-zhang,li-da-tong,SteinerGeodetic,pelayo-comment}.
For instance, \cite{pelayo-1} was dedicated to obtain some families
of graphs in which every Steiner set is a geodetic set, but the
problem of characterizing such a graphs remains open.

In this work we show some classes of graph in which every Steiner
set is a geodetic set. For instance, we prove that if $G$ is a graph
with diameter two, then every Steiner set of $G$ is also a geodetic
set. We also obtain some relationships between the Steiner (geodetic)
sets of corona product graphs and the Steiner (geodetic) sets of its
factors and, as a consequence of this study, we obtain that if  $G$ is a corona product graph, then  every Steiner set of $G$ is a
geodetic set.

We begin by stating some terminology and notation. In this paper $G=
(V,E)$ denotes a connected simple graph of order $n=|V|$.  We denote two adjacent vertices $u$ and $v$ by $u\sim v$. Given
a set $W\subset V$ and a vertex $v\in V$, $N_W(v)$ represents the
set of neighbors that $v$ has in $W$, i.e. $N_W(v)=\{u\in
W\,:\,u\sim v\}$. The subgraph induced by a set $W\subset V$ will be denoted by $\langle W\rangle$.

The distance $d_G(u,v)$ between two vertices $u$ and
$v$ is the length of a shortest $u-v$ path in $G$. If there is no ambiguity,  we will use the notation $d(u,v)$ instead of $d_G(u,v)$.
A shortest $u-v$ path is called  $u-v$ \textit{geodesic}.
We define $I_G[u,v]$\footnote{If there is no ambiguity, then we will
use $I[u,v]$.} to be the set of all vertices lying on some $u-v$
geodesic of $G$, and for a nonempty set $S\subseteq V$,
$I_G[S]=\bigcup_{u,v\in S}I_G[u,v]$ ($I[S]$ for short). A set
$S\subseteq V$ is a \textit{geodetic set} of $G$ if
$I_G[S]=V$ and a geodetic set of minimum cardinality is called a
\textit{minimum geodetic set} \cite{HaLoTs}. The cardinality of a minimum geodetic set of $G$ is called the \textit{geodetic number} of $G$ and it is
denoted by $g(G)$. A vertex $v\in V$ is \emph{geodominated} by a pair
$x,y\in V$ if $v$ lies on  an $x-y$ geodesic of $G$. For an integer
$k\ge 2$, a vertex $v$ of a graph $G$ is $k$-\emph{geodominated} by a pair
$x,y$ of vertices in $G$ if $d(x,y)=k$ and $v$ lies on an $x-y$
geodesic of $G$. A subset $S\subseteq V$ is a \textit{$k$-geodetic
set} if each vertex $v$ in $\overline{S}=V - S$ is
$k$-geodominated by some pair of vertices of $S$. The minimum
cardinality of a $k$-geodetic set of $G$ is its \textit{$k$-geodetic
number} $g_k(G)$. It is clear that $g(G)\leq g_k(G)$ for every $k$.

For a nonempty set $W$ of vertices of a connected graph, the \emph{Steiner
distance} $d(W)$ of $W$ is the minimum size of a connected subgraph
of $G$ containing $W$ \cite{chartrand-first}. Necessarily, such a
subgraph is a tree and it is called a \emph{Steiner tree} with respect to
$W$ or a Steiner $W$-tree, for short. For a set $W\subseteq V$, the
set of all vertices of $G$ lying on some Steiner $W$-tree is denoted
by $S_G[W]$ (or by $S[W]$, if there is no ambiguity). If $S_G[W]=V$,
then $W$ is called a Steiner set of $G$. The \emph{Steiner number} of a graph $G$,
denoted by $s(G)$, is the minimum cardinality among the Steiner sets of $G$.

Let $G$ and $H$ be two graphs and let $n$ be the order of $G$.
The corona product $G\odot H$ is defined as the graph obtained from
$G$ and $H$ by taking one copy of $G$ and $n$ copies of $H$ and
then joining by an edge, all the vertices from the $i^{th}$-copy of
$H$ with the $i^{th}$-vertex of $G$. Throughout the article we will
denote  by $V=\{v_1,v_2,...,v_{n}\}$   the set of vertices of $G$ and by $H_i=(V_i,E_i)$   the copy of $H$ in $G\odot H$ such that $v_i\sim v$ for every $v\in V_i$.

\section{Geodetic number of corona product graphs}

We begin by stating some results that we will use as tool in this section. The first one is the following well-known result.

\begin{lemma}{\em \cite{HaLoTs}}\label{geo-kn}
Let $G$ be a connected graph of order $n$. Then $g(G)=n$ if and only
if $G\cong K_n$.
\end{lemma}

Our second tool will be the following useful lemma related to the geodetic
sets of corona product graphs.

\begin{lemma}\label{lemma-principal-geo}
Let $G=(V,E)$ be a connected graph of order $n$ and let $H$ be a
graph. Let $H_1=(V_1,E_1),H_2=(V_2,E_2),...,H_{n}=(V_{n},E_{n})$ be the
$n$ copies of $H$ in $G\odot H$.
\begin{itemize}
\item[{\rm (i)}] Given three different vertices $a$, $b$ and $v$ of $G\odot H$, if $v\in V_i$ and $(a\notin V_i$ or $b\notin V_i)$, then $v\notin I_{G\odot H}[a,b]$.
\item[{\rm (ii)}] If $W$ is a geodetic set of $G\odot H$, then $W\cap V_i\ne \emptyset$, for every $i\in \{1,...,n\}$.
\item[{\rm (iii)}] If $W$ is a minimum geodetic set of $G\odot H$ and either $n\ge 2$ or $(n=1$ and $H$ is a non-complete graph$)$,
then $W\cap V=\emptyset$.
\item[{\rm (iv)}] If $H$ is a non-complete graph and $W$ is a minimum geodetic set of $G\odot H$, then for every
$i\in \{1,...,n\}, $ $W_i=W\cap V_i$ is a geodetic set of
$\langle v_i\rangle \odot H_i$.
\end{itemize}
\end{lemma}

\begin{proof} (i) and (ii) follow directly from the fact that the vertices belonging to  $V_i$ are adjacent to only one vertex not in $V_i$.

Now let $W'$ be a geodetic set of  $G\odot H$ and let $W=W'-V$. We will show that $W$ is a geodetic set of
$G\odot H$.  By (ii) we have that for every $i\in
\{1,...,n\}$ it is satisfied, $W\cap V_i\ne \emptyset$, and by
(i), we have that if $v\in V_i$, then  there exist $a',b'\in V_i\cap W$,
such that $v\in I_{G\odot H}[a',b']$. Now, if $n\ge 2$, then for
every vertex $v_i\in V$ we have that $v_i\in I_{G\odot H}[c,d]$,
with $c\in W\cap V_i$ and $d\in W\cap V_j$, $j\ne i$. Thus, $W$
is a geodetic set of $G\odot H$. On the other hand, if $n=1$ and
$H$ is a non-complete graph, then $G\odot H$ is a non-complete graph and, by Lemma \ref{geo-kn},  $g(G\odot H)\le n_2$, where $n_2$ is the order of $H$. Hence, $\langle W\rangle$ is
not isomorphic to a complete graph. So, there exist two vertices $x,y\in W$ such that the vertex of $G\cong K_1$ belongs to $
I_{G\odot H}[x,y]$. Moreover, by (i), every vertex of $\overline{W}$ different from the vertex of $G$ is
geodominated by two vertices of $W$.  Thus,  $W$ is a
geodetic set of $G\odot H$ and, as a consequence, (iii) follows.

Finally, let $H$ be a non-complete graph and let $W$ be   a minimum geodetic set of $G\odot H$. By (ii) we have that $W_i=W\cap V_i\ne \emptyset$.  Also, by (iii) we have that $V\cap W=\emptyset$. Now we suppose that $W_i$ is not a geodetic set of $\langle v_i\rangle \odot H_i$.
Hence, there exists $v\in V_i\cup \{v_i\}$ such that $v\notin
I_{\langle v_i\rangle \odot H_i}[x,y]$ for every $x,y\in W_i$. By (i) we have that if $v\in V_i-W$, then $v$ must be geodominated by vertices of $W_i$, which is a contradiction, so $v\not\in V_i$, i.e., $v=v_i$. Now, since
$v_i$ is adjacent to every vertex of $H_i$ and $H_i$ is a non-complete graph, we obtain that there exist two non-adjacent vertices $c,d$ of $H_i$ such that  $c,d\in  W_i$. Hence, $v_i\in I_{\langle v_i\rangle \odot H_i}[c,d]$, a contradiction. Therefore, (iv) follows.
\end{proof}

The following relation between $g(H)$ and $g(K_1\odot H)$, which we will
use here, was obtained in \cite{pelayo-1}.

\begin{lemma}{\em \cite{pelayo-1}}\label{geo-H-K1+H}
For any graph $H$, $g(K_1\odot H)\ge g(H)$.
\end{lemma}

A vertex $v$ is an extreme vertex in a graph
$G$ if the subgraph induced by its neighbors is complete.
The following lemma is a consequence of the observation
that each extreme vertex $v$ of $G$ is either the initial or
terminal vertex of a geodesic containing $v$.

\begin{lemma}{\em \cite{ChHaZh3}}\label{lemmaExtreme}  Every geodetic set of a graph contains its
extreme vertices.
\end{lemma}

\begin{proposition}\label{general-bound-geo}
Let $G$ be a connected graph of order $n_1$ and let $H$ be a graph
of order $n_2$. If $n_1\ge 2$ or $(n_1=1$ and $H$ is a non-complete graph$)$, then $$n_1g(H)\le g(G\odot H)\le n_1n_2.$$
 The upper bound is achieved if and only if $H$ is isomorphic
to a graph in which every connected component is isomorphic to a
complete graph.

Moreover, if no connected component of $H$ is isomorphic to a
complete graph, then $$g(G\odot H)\le n_1(n_2-1).$$
\end{proposition}

\begin{proof}
If $H\cong K_{n_2}$,   the vertices of the set $\cup_{i=1}^{n_1} V_i$ are extreme vertices. Then, by Lemma \ref{lemmaExtreme} we have $g(G\odot K_{n_2})\ge n_1n_2=n_1g(K_{n_2}).$  For non-complete graphs the lower bound follows directly from Lemma \ref{lemma-principal-geo}  (iv)
and Lemma \ref{geo-H-K1+H}. On the other hand, if $n_1\ge 2$, then every vertex
$v_i\in V$ is geodominated, in $G\odot H$, by two vertices belonging
to different copies of $H$. So, the set $\bigcup_{i=1}^{n_1}V_i$ is
a geodetic set of $G\odot H$. Thus, $g(G\odot H)\le n_1n_2$. Finally, if $n_1=1$, then the order of $G\odot H$ is $n_2+1$. Hence, if $H$ is a non-complete graph, then Lemma \ref{geo-kn} leads to the upper bound  $g(K_1\odot H)\le n_2$.

 Now, let us suppose that
there is a component of $H$ which is not isomorphic to a complete
graph. In such a case, there are three different vertices $u_i,x_i,y_i\in V_i$
such that $u_i\in I_{H_i}[x_i,y_i]$, with $i\in \{1,...,n_1\}$. Let $V=\{v_1,...,v_{n_1}\}$,
$U_i=V_i-\{u_i\}$, with $i\in \{1,...,n_1\}$, and let
$U=\bigcup_{i=1}^{n_1}U_i$. We will show that $U$ is a geodetic set of $G\odot H$. Since for every vertex $u_i\in \overline{U}_i$ we have
that $u_i\in I_{H_i}[x_i,y_i]$, we obtain that $u_i\in I_{G\odot
H}[U]$. Also, as for every $v_i\in V$, we have that $v_i\in
I_{G\odot H}[a,b]$, for some $a\in U_i$ and $b\in U_j$, with $i\ne
j$, we obtain that $v_i\in I_{G\odot H}[U]$. Therefore, $U$ is a
geodetic set of $G\odot H$ and, as a consequence, $g(G\odot H)\le |U|=n_1(n_2-1)$. Therefore, if $g(G\odot H)=n_1n_2$, then $H$
is isomorphic to a graph in which every connected component is
isomorphic to a complete graph.
\end{proof}

\begin{theorem}\label{geodetic-corona-suma}
Let $G$ be a connected graph of order $n$ and let $H$ be a non-complete graph. Then, $$g(G\odot H)=ng(K_1\odot H).$$
\end{theorem}

\begin{proof}
Let $W$ be a minimum geodetic set of $G\odot H$. From
Lemma \ref{lemma-principal-geo} (iii) we have that $W\cap V=\emptyset$. Also, by Lemma
\ref{lemma-principal-geo} (ii) and (iv) we have that for every $i\in
\{1,...,n\}$, the set $W_i=W\cap V_i\ne \emptyset$ is a geodetic set of
$\langle v_i\rangle \odot H_i\cong K_1\odot H$.   Hence, we have
$$g(G\odot H)=|W|=\sum_{i=1}^{n}|W_i|\ge \sum_{i=1}^{n}g(\left\langle v_i\right\rangle \odot H_i)=ng(K_1\odot H).$$

On the other hand, let $U_i\subset V_i\cup \{v_i\}$ be a minimum geodetic set of
$\left\langle v_i\right\rangle \odot H_i$ and let $U=\cup_{i=1}^{n}U_i$. Notice that, by Lemma \ref{lemma-principal-geo} (iii),
$v_i\not\in U_i$.  We will show that $U$ is a geodetic set of
$G\odot H$. Let us consider a vertex $x$ of $G\odot H$. We have the
following cases.

Case 1: If $x\in (V_i\cup\{v_i\})-U_i$, then there exist $u,v\in
U_i$ such that $x\in I_{\left\langle v_i\right\rangle \odot H_i}[u,v]$.  So, $x\in I_{G\odot
H}[u,v]$.

Case 2: If $x=v_i\in V$ and $n\ge 2$, then for every vertex $v\in
U_i$ and some $u\in U_j$, $j\ne i$ we have that $x\in I_{G\odot
H}[u,v]$. Also, if $x\in V$ and $n=1$, then as $H$ is a non-complete graph, there exist two different vertices $a,b\in U=U_1$, such that
$x\in I_{G\odot H}[a,b]$.

Thus, every vertex $x$ of $G\odot H$ is geodominated by a pair of
vertices of $U$ and, as a consequence, $g(G\odot H)\le ng(K_1\odot H)$.
Therefore, we obtain that $g(G\odot H)= ng(K_1\odot H)$.
\end{proof}

The geodetic number of wheel graphs and fan graphs were studied in
\cite{pelayo-1} and \cite{canoy}.

\begin{remark}{\rm \cite{pelayo-1}}\label{geo-wheel}
If $n\ge 4$, then
$g(W_{1,n})={\left\lceil{\frac{n}{2}}\right\rceil}$.
\end{remark}

\begin{remark}{\rm \cite{pelayo-1,canoy}}\label{geo-fan}
If $n\ge 3$, then
$g(F_{1,n})={\left\lceil{\frac{n+1}{2}}\right\rceil}$.
\end{remark}

As a particular cases of Theorem \ref{geodetic-corona-suma} and by
using the above remarks we obtain the following results.

\begin{corollary}
Let $G$ be a connected graph of order $n_1$.
\begin{itemize}
\item[{\rm (i)}]If $n_2\ge 4$, then $g(G\odot
C_{n_2})=n_1g(W_{1,n_2})=n_1\left\lceil\frac{n_2}{2}\right\rceil$.
\item[{\rm (ii)}]If $n_2\ge 3$, then $g(G\odot
P_{n_2})=n_1g(F_{1,n_2})=n_1\left\lceil\frac{n_2+1}{2}\right\rceil$.
\end{itemize}
\end{corollary}

>From Lemma \ref{geo-H-K1+H} we have that $g(K_1\odot H)\ge g(H)$. Hence,
Theorem \ref{geodetic-corona-suma} leads to the lower bound of Proposition
\ref{general-bound-geo}. Now we are interested in those graphs in which
$g(H)=g(K_1\odot H)$.

\begin{theorem}\label{geodetic-iff-2-geo}
For a connected graph  $H$, the following statements are equivalent:
\begin{itemize}
\item   $g(H)=g(K_1\odot H)$.
\item $g(H)=g_{_2}(H)$.
\end{itemize}
\end{theorem}

\begin{proof} Let us suppose $g(H)=g_{_2}(H)$. Let $W$ be a $2$-geodetic set of minimum cardinality in $H$. Hence, for every vertex $u\in
\overline{W}$ there exist $a,b\in W$, such that $u\in I_H[a,b]$ and
$d_H(a,b)=2$. Since every geodesic of length two in $H$ is a
geodesic in $K_1\odot H$, we have that $W$ is a geodetic set of $K_1\odot H$.
As a consequence, $g(H)\ge g(K_1\odot H)$. Hence, by Lemma \ref{geo-H-K1+H}
we conclude that $g(H)=g(K_1\odot H)$.

On the other hand, let us suppose $g(H)=g(K_1\odot H)$.  Let $U$ be a
minimum geodetic set of $K_1\odot H$ and let $v$ be the
vertex of $K_1$. Since  $H$ can not be a complete graph, by Lemma \ref{lemma-principal-geo} (iii) we have
that $v\notin U$.
Now, since $K_1\odot H$ has diameter two, we have that for
every vertex $u$ of $H$ not belonging to $U$, there exist $a,b\in U$ such that  $u\in
I_{K_1\odot H}[a,b]$ and $d_{H}(a,b)=2$ (Note that if $d_{H}(a,b)>2$, then $u\not\in I_{K_1\odot H}[a,b]=\{a,b,v\}$). Hence,  $U$ is a $2$-geodetic set of $H$. Thus, $g_{_2}(H)\le
|U|=g(K_1\odot H)=g(H)$. Also, as  $g(H)\le
g_{_2}(H)$, we obtain that $g(H)=g_{_2}(H)$.
\end{proof}

\begin{theorem}\label{Thgeodetic2geodetic}
Let $G$ be a connected graph of order $n$ and let $H$ be a connected non-complete graph. Then the following statements are equivalent:
\begin{itemize}
\item $g(G\odot H)=ng(H)$.
\item $g(H)=g_{_2}(H)$.
\end{itemize}
\end{theorem}
\begin{proof}
The result is a direct consequence of Theorem \ref{geodetic-corona-suma} and Theorem \ref{geodetic-iff-2-geo}.
\end{proof}
Since for every graph $H$ of diameter two we have $g(H)=g_{_2}(H)$, Theorem \ref{Thgeodetic2geodetic} leads to the following result.

\begin{corollary}\label{geodeticDiam2GH=n1H-Corollary}
Let $G$ be a connected graph of order $n$ and let $H$ be a
graph. If $D(H)=2$, then $$g(G\odot
H)=ng(H).$$
\end{corollary}

Another consequence of Theorem \ref{geodetic-iff-2-geo} is the following result.

\begin{corollary}
Let $G$ and $H$ be two connected graphs of order $n_1$ and  $n_2$, respectively. Let $N_k$ be the empty graph of order $k\ge 2$. Then
$$g(G\odot (H\odot N_k))=n_1n_2k.$$
\end{corollary}
\begin{proof}
The result follows from the fact that $g(H\odot N_k)=g_{_2}(H\odot N_k)=n_2k$. That is, the set composed by the $n_2k$ pendant vertices of $H\odot N_k$ form a geodetic set of $H\odot N_k$  which is a 2-geodetic set. So, $g(H\odot N_k)\le g_{_2}(H\odot N_k)\le n_2k$. Moreover, since every pendant vertex is an extreme vertex, by Lemma \ref{lemmaExtreme} we have $g(H\odot N_k)\ge n_2k$. Therefore, the result follows.
\end{proof}

The following result improves the lower bound in Proposition \ref{general-bound-geo} for those graphs whose geodetic number is different from its 2-geodetic number.

\begin{theorem}
Let $G$ be a connected graph of order $n$ and let $H$ be a non-complete graph.
If $g(H)\ne g_{_2}(H)$, then $$g(G\odot H)\ge n\left( g(H)- 1\right).$$
\end{theorem}

\begin{proof}
As a direct consequence of Theorem \ref{geodetic-iff-2-geo} and Lemma \ref{geo-H-K1+H} we obtain that, if $g(H)\ne g_{_2}(H)$, then
\begin{equation}\label{eqIntermedia}
g(K_1\odot H)\ge  g(H)- 1.
\end{equation} Hence, the result follows directly by Theorem \ref{geodetic-corona-suma} and (\ref{eqIntermedia}).
\end{proof}

\section{Steiner number of corona product graphs}

In this section the main tool will be the following basic lemma.

\begin{lemma}\label{lema-steiner-new}
Let $G=(V,E)$ be a connected graph of order $n_1$ and let $H$ be a graph
of order $n_2$. Let $H_1=(V_1,E_1),H_2=(V_2,E_2),...,H_{n}=(V_{n},E_{n})$ be the
$n_1$ copies of $H$ in $G\odot H$.
\begin{enumerate}[{\rm (i)}]
\item If $n_1\ge 2$ and $A\subseteq \cup_{i=1}^{n_1}V_i$ with $A\cap V_i\ne \emptyset$, for every $i\in \{1,...,n_1\}$, then every Steiner $A$-tree contains all vertices of $G$
\item If $U$ is a Steiner set of $G\odot H$, then $U\cap V_i\ne \emptyset$, for every $i\in \{1,...,n\}$.
\item If $n_1\ge 2$ or $n_2\ge 2$, then for every Steiner set $U$ of minimum cardinality in $G\odot H$ it follows $U\cap V=\emptyset$.
\end{enumerate}
\end{lemma}

\begin{proof}
(i) follows from the fact that if there exists a Steiner $A$-tree $T$ not containing a vertex of $G$, then $T$ is not connected, which is a contradiction. (ii) follows directly from the fact that the vertices belonging to $V_i$ are adjacent to only one vertex not in $V_i$.

Now let $U'$ be a Steiner set of $G\odot H$ and let $U=U'-V$. We will show that $U$ is a Steiner set for
$G\odot H$. By (ii) we have that $U\cap V_i\ne \emptyset$, for every $i\in \{1,...,n\}$. Also, if $v\in V_i$, then we have that there exists a
Steiner $U$-tree in $G\odot H$ such that it contains the vertex $v$.  Now, since $n_1\ge 2$ we obtain that
every vertex $v_i\in V$ belongs to every Steiner $U$-tree (note that every shortest $u-v$ path, where $v\in V_i$ and  $u\in V_j$, $j\ne i$, must contain $v_i$). Thus, $U$ is a Steiner set for
$G\odot H$ and (iii) follows.
\end{proof}

The next lemmas obtained in \cite{chartrand-zhang} will be useful to obtain
our results.

\begin{lemma}{\em \cite{chartrand-zhang}}\label{steiner-kn}
Let $G$ be a connected graph of order $n$. Then $s(G)=n$ if and only
if $G\cong K_n$.
\end{lemma}

Before present our main results about the Steiner number, let us show the following useful lemma.

\begin{lemma}\label{lemma+K1Steiner}
For any graph $G$, $s(K_1\odot G)\ge s(G)$.
\end{lemma}

\begin{proof}
Let $n$ be the order of $G$. If $G\cong K_n$, then $K_1\odot G\cong K_{n+1}$, so by Lemma \ref{steiner-kn}, $s(K_1\odot G)=n+1>n=s(G).$ If $G\not\cong K_n$, then the result follows immediately from Lemma \ref{lema-steiner-new} (iii).
\end{proof}

\begin{proposition}\label{general-bound-stein}
Let $G=(V,E)$ be a connected graph of order $n_1$ and let $H$ be a graph
of order $n_2$. If $n_1\ge 2$, then $s(G\odot H)=n_1n_2.$
\end{proposition}

\begin{proof}
Let $A=\cup_{i=1}^{n_1}V_i$. By Lemma \ref{lema-steiner-new} (iii) we have that every  Steiner set of minimum cardinality is a subset of $A$. Thus, $A$ is a Steiner set of $G\odot H$ and, as a consequence, $s(G\odot H)\le n_1n_2.$

Now, let us suppose $B$ is a Steiner set of minimum cardinality in $G\odot H$. By Lemma \ref{lema-steiner-new} (iii) we have that $B$ does not contain any vertex of $G$. Now, let us suppose there exists a vertex $v_i\in V$ such that $B\cap V_i\subsetneq V_i$. Let $B_i=B\cap V_i$ and let $u\in V_i-B_i$. Since every vertex of $B_i$ is adjacent to $v_i$, and $v_i$ belongs to every Steiner $B$-tree $T$, we have that the size of the restriction of  $T$ to $V_i\cup\{v_i\}$ is $|B_i|$.  Thus, the vertex $u$ does not belong to any Steiner $B$-tree in $G\odot H$, which is a contradiction. Thus, for every $i\in \{1,...,n_1\}$ we have that $B\cap V_i=V_i$. Therefore, $s(G\odot H)\ge n_1n_2$. The proof is complete.
\end{proof}

The Steiner number of wheel graphs and fan graphs were studied in
\cite{pelayo-1} and \cite{canoy}.

\begin{remark}{\rm \cite{pelayo-1}}\label{stein-wheel}
If $n\ge 4$, then $s(W_{1,n})=n-2$.
\end{remark}

\begin{remark}{\rm \cite{pelayo-1,canoy}}\label{stein-fan}
If $n\ge 3$, then $g(F_{1,n})=n-1$.
\end{remark}

\begin{theorem}\label{steinerK1+H-iff-diam2}
Let $H$ be a connected non complete graph. Then the following statements are equivalent:
\begin{itemize}
\item $s(K_1\odot H)=s(H)$.
\item $D(H)=2$.
\end{itemize}
\end{theorem}

\begin{proof}
Let $B$ be a Steiner set of minimum cardinality in $H$ and let $v$ be the vertex of $K_1$. If $D(H)=2$,
then there exist three vertices of $H$ such that  $x,y\in B$ and $z\not \in B$,  $d_H(x,y)=2$ and $x,y\in N_B(z)$.
So, if we take a Steiner $B$-tree $T$ in $H$ containing the path $xzy$, then replacing the vertex $z$ of  $T$ by the vertex $v$, and replacing every edge $uz$ of T by a new edge $uv$,  we obtain a Steiner $B$-tree $T'$ in $K_1\odot H$. Hence, $B$ is a Steiner set for $K_1\odot H$. Therefore, $s(H)\ge s(K_1\odot H)$ and, by Lemma \ref{lemma+K1Steiner}, we conclude  $s(H)= s(K_1\odot H)$.

Now, let $H$ be a graph such that $s(K_1\odot H)=s(H)$. Let $W$ be a Steiner set of minimum cardinality in $K_1\odot H$ and let $v$ be the vertex of $K_1$. We first show that $W$ is a Steiner set for $H$. Note that by Lemma \ref{lema-steiner-new} (iii), $v\not\in W$. Since the star graph of center $v$ is a Steiner $W$-tree, we have that the Steiner distance of $W$ in $K_1\odot H$ is $d(W)=|W|$. If $\langle W\rangle$ is connected,
then $|W|$ is the order of $K_1\odot H$, which is a contradiction. Thus,
$\langle W\rangle$ is non connected. Let $\langle W_1\rangle$, $\langle W_2\rangle$, ...,$\langle W_k\rangle$ be the connected components of $\langle W\rangle$.  If there exists a vertex $u$ of $H$ such that $u\not \in W$ and $N_{W_i}(u)=\emptyset$, for some $i$, then the Steiner distance of $W$ in $K_1\odot H$ is $d(W)>\sum_{i=1}^k|W_i|=|W|$, which is a contradiction. So,  every vertex $u$ of $H$ not belonging to $W$ is at distance one to every connected component of $\langle W\rangle$ and, as a consequence, $W$ is a Steiner set of $H$,  which has minimum cardinality since $s(K_1\odot H)=s(H)$.
Let us show that $D(H)=2$. On the contrary, we suppose that $D(H)\ge 3$ (note that $H$ is not a complete graph).
>From the assumption $D(H)\ge 3$, we conclude that for each vertex $u$ of $H$, not belonging to $W$, there exist  $y\in W_i$ (for some $i$) such that $d(y,u)=2$. Let $x\in  W_i$ be a neighbor of both $u$ and $y$,  and let $W'=W-\{x\}$. Then we have that every Steiner $W$-tree of $H$ is a Steiner $W'$-tree of $H$ and, as a consequence,  $W'$ is a Steiner set of  $H$, which is a contradiction. Therefore, $D(H)=2$.
\end{proof}

\section{Relationships  between the geodetic number and the Steiner number}

Here we show  some classes of graphs where the Steiner number is greater than or equal to the geodetic number.

\begin{theorem}\label{geodetic-steiner-diam-2}
If $G$ is a graph of diameter two, then every Steiner set for $G$ is a geodetic set for $G$.
\end{theorem}

\begin{proof}
Let $W$ be a Steiner set of minimum cardinality in $G$ and let $n$ be the order of $G$. If $\langle W\rangle$ is connected,
then $|W|=n$. So, by Lemma \ref{steiner-kn} we have that $G\cong
K_n$, which is a contradiction because $G$ has diameter two. Thus,
$\langle W\rangle$ is non connected. Let $\langle  B_1\rangle,\langle B_2\rangle,...,\langle B_r\rangle$ be the
connected components of $W$. We assume that $W$ is not a geodetic set.
Then there exists a vertex $x$ of $G$ such that $x\not\in I[W]$.
Thus, $x\notin W$ and $x\notin I[u,v]$ for every $u,v\in W$. Hence,
$N_W(x)\subseteq B_i$, for some $i\in \{1,...,r\}$. Since $G$ has
diameter two, any Steiner $W$-tree is formed by $r$ Steiner
$B_i$-trees connected by vertices $v_1,v_2,...,v_t$, $t\ge 1$, not
belonging to $W$ such that $N_W(v_i)\not\subset B_j$, for every
$i\in \{1,...,t\}$ and $j\in \{1,...,r\}$. Hence,
$S[W]=\left(\bigcup_{i=1}^rB_i\right)\cup\left(\bigcup_{i=1}^t\{v_i\}\right)$.
Therefore $x\not\in S[W]$, which is a contradiction.
\end{proof}

\begin{corollary}\label{coro-geodetic-steiner-diam-2}
If $G$ is a graph of diameter two, then $g(G)\le s(G)$.
\end{corollary}

Now, from Theorem \ref{geodetic-corona-suma}, Proposition
\ref{general-bound-stein} and Corollary
\ref{coro-geodetic-steiner-diam-2} we obtain the following
interesting result in which we give an infinite number of graphs $G$
satisfying that $g(G)\le s(G)$.

\begin{theorem}
Let $G$ be a connected graph of order $n_1\ge 2$ and let $H$ be any
non complete graph of order $n_2$. Then, $$g(G\odot H)\le s(G\odot H).$$
\end{theorem}

\begin{proof}
By Theorem \ref{geodetic-corona-suma} we have that $g(G\odot
H)=ng(K_1\odot H)$. Since, $K_1\odot H$ has diameter two, by using Corollary
\ref{coro-geodetic-steiner-diam-2} we have that $g(K_1\odot H)\le
s(K_1\odot H)$. Finally, by Proposition \ref{general-bound-stein} we know that $s(G\odot H)=n_1n_2$. Hence,
$$g(G\odot H)=n_1g(K_1\odot H)\le n_1s(K_1\odot H)\le n_1n_2=s(G\odot H).$$
\end{proof}

\noindent{\large\bf Acknowledgments}

We would like to thank Dr. Ayyakutty Vijayan for his helpful suggestions and comments.

\end{document}